\documentclass[12pt]{amsart}
\usepackage[english]{babel}
\usepackage{amsmath,amssymb,enumerate,amsthm}
\usepackage[T1]{fontenc}
\usepackage{cite}
\usepackage{latexsym}
\usepackage{setspace}
\usepackage{bm}
\usepackage{qtree} 
\usepackage[linguistics]{forest}
\usepackage{a4}
\usepackage{mathtools}
\newtheorem{lemma}{Lemma}[section]
\newtheorem{theorem}[lemma]{Theorem}

\newtheorem{coro}[lemma]{Corollary}

\newtheorem{rk}[lemma]{Remark}

\newtheorem{example}[lemma]{Example}

\def\C{\mathbb{C}}
\def\R{\mathbb{R}}
\def\Z{\mathbb{Z}}

\def\C{\mathcal{C}}
\def\T{\mathcal{T}}

\def\GL{\mathrm{GL}(2,\mathbb{Z})}
\def\PGL{\mathrm{PGL}(2,\mathbb{Z})}
\def\PSL{\mathrm{PSL}(2,\mathbb{Z})}
\def\={\;=\;}
\def\.={\;\dot{=}\;}
\def\Im{\mathrm{Im}}
\def\Re{\mathrm{Re}}
\def\l{\ell}
\def\H{\mathcal{H}}

\def\c{\bm c}

\def\p{\bm p}
\def\q{\bm q}

\DeclarePairedDelimiter\ceil{\lceil}{\rceil}
\DeclarePairedDelimiter\floor{\lfloor}{\rfloor}

\begin{document}

\title[Cycle integrals of modular functions  and  Markov geodesics]{Cycle integrals of modular functions, Markov geodesics and a conjecture of Kaneko}

\author{P. Bengoechea}
\address{ ETH, Mathematics Dept.
\\CH-8092, Z\"urich, Switzerland}
\email{paloma.bengoechea@math.ethz.ch}
\thanks{Bengoechea's research is supported by SNF grant 173976.}

\author{ \"O. Imamoglu}
\address{ ETH, Mathematics Dept.
\\CH-8092, Z\"urich, Switzerland}
\email{ ozlem@math.ethz.ch}
\maketitle

\begin{abstract} In this paper we study the values of modular functions at the Markov quadratics which are defined in terms of their cycle integrals   along the associated closed geodesics. These numbers are shown to satisfy  two properties  that were conjectured by Kaneko. More precisely we show that  the values of a modular function $f$,   along any branch $B$ of the Markov tree, converge to the value of $f$ at the Markov number which is the predecessor of the  tip of $B$.  We also prove an interlacing property for these values.

\end{abstract}  
\section{Introduction}


  A well known theorem of Dirichlet   asserts that for any irrational number $x$, there are infinitely many rational numbers $p/q$ satisfying $|x-\frac{p}{q}|<\frac{1}{q^2}$. For irrational numbers that are algebraic, thanks to a theorem of Roth (\cite{Roth}), the exponent $2$ is optimal. The constant factor, on the other hand, can be improved and    a classical theorem of Hurwitz asserts that  for every irrational number  $x$ there exist infinitely    many rational numbers $p/q$  satisfying
$$
\left|x-\dfrac{p}{q}\right|<\dfrac{1}{\sqrt{5}q^2}.
$$
The constant $1/\sqrt{5}$ is best possible but if we exclude as $x$ the numbers that are $\PGL$-equivalent to the golden ratio $(1+\sqrt{5})/2$, the constant $1/\sqrt{5}$ improves to $1/\sqrt{8}$. If we also exclude the numbers that are $\PGL$-equivalent to $\sqrt{2}$, then the constant  improves to $5/\sqrt{221}$. By proceeding in this way, one obtains the Lagrange spectrum defined by
$$
L:=\left\{\nu(x)\right\}_{x\in\R}\subseteq\left[0,1/\sqrt{5}\right]\quad\mbox{with}\quad
\nu(x)=\liminf_{q\rightarrow\infty} q\|qx\|,
$$
 where $\|x\|$ denotes the distance from a real number $x$
to a closest integer.
The quantity $\nu(x)$  provides a measure of approximation of $x$ by the rationals.
 For almost all $x\in\R$ we have $\nu(x)=0$ and when $\nu(x) >0$ we call $x$ badly approximable. Real quadratic irrationals are badly approximable, the worst ones being the golden ratio and its $\PGL$-equivalents, followed by $\sqrt{2}$ and its $\PGL$-equivalents, etc. The Lagrange spectrum is not discrete (cf \cite{Hall}) but the part of the spectrum  in the subinterval $(1/3,1/\sqrt{5}]$  corresponding to classes of worst irrational numbers is,  with $1/3$ as its only accumulation point. $L\cap(1/3,1/\sqrt{5}]$  is well understood thanks to the work of Markov (cf \cite{Mar1}, \cite{Mar2}) which   connects this question of Diophantine approximation to the Diophantine equation 
  \begin{equation}\label{markovn}
x^2+y^2+z^2=3xyz.
\end{equation}
  
 The set of Markov triples comprising the positive integer solutions $(x,y,z)$ of (\ref{markovn}) can be obtained starting with $(1,1,1)$, $(1,1,2)$,
   $(2,1,5)$ and then proceeding recursively going from $(x,y,z)$ to the new triples obtained by Vieta involutions $(z,y,3yz-x)$ and $(x,z,3xz-y)$.
   The Markov numbers  are the greatest coordinates of Markov triples. They form the  Markov sequence
      $$
\left\{m_i\right\}_{i=1}^\infty=\left\{1,2,5,13,29,34,89,169,194,\ldots\right\}.
$$
The   Markov number $m_i$ is  associated   to a  quadratic irrationality
$$
\theta_i=\dfrac{3m_i-2k_i+\sqrt{9m_i^2-4}}{2m_i},
$$
where $k_i$ is an integer that satisfies $a_ik_i\equiv b_i\pmod{m_i}$ and $(a_i,b_i,m_i)$ is a solution to \eqref{markovn} with $m_i$ maximal. Since $k_i$ is uniquely defined modulo $m_i$, $\theta_i$ is uniquely defined modulo 1. 
Markov showed that 
$ 
\nu(\theta_i)=\sqrt{9-4/m_i^2},
$ 
and
$ 
L\cap(1/3,1/\sqrt{5}]=\left\{\nu(\theta_i)\right\}_{i\geq 1}.
$
Moreover, any $x\in\R$ for which $\nu(x)\in L\cap(1/3,1/\sqrt{5}]$ is $\PGL$-equivalent to a Markov quadratic $\theta_i$.

Markov numbers come with a tree structure, inherited from Vieta involutions, that arranges them as below:

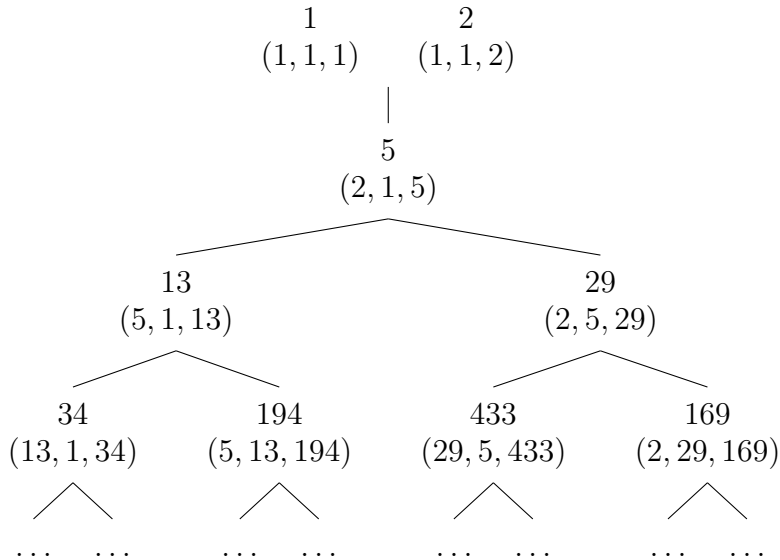
\begin{center}
\begin{forest}
[
{$\begin{array}{c}1\\(1,1,1)\end{array}$\quad$\begin{array}{c}2\\(1,1,2)\end{array}$}
[
{$\begin{array}{c}5\\(2,1,5)\end{array}$} 
[
{$\begin{array}{c}13\\(5,1,13)\end{array}$} 
[
{$\begin{array}{c}34\\(13,1,34)\end{array}$}
[$\ldots$]
[$\ldots$]]
[
{$\begin{array}{c}194\\(5,13,194)\end{array}$}
[$\ldots$]
[$\ldots$]
]] 
[
{$\begin{array}{c}29\\(2,5,29)\end{array}$} 
[
{$\begin{array}{c}433\\(29,5,433)\end{array}$}
[$\ldots$]
[$\ldots$]] 
[
{$\begin{array}{c}169\\(2,29,169)\end{array}$}
[$\ldots$]
[$\ldots$]]]]]
\end{forest}
\end{center}

Here $(a,b,c)$ is a solution to \eqref{markovn}.
The Markov quadratics inherit the same tree structure which  can be given in terms of their continued fractions as
%

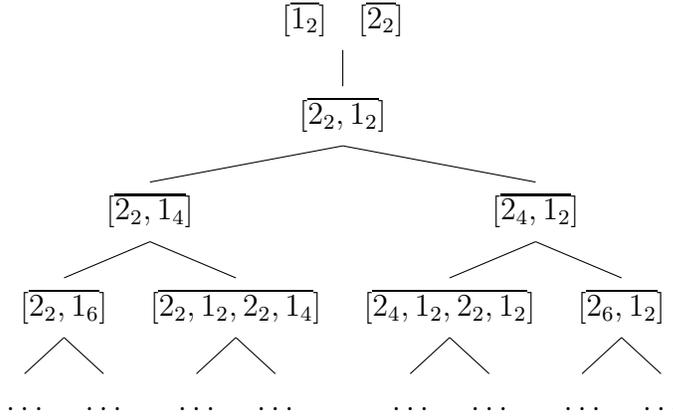
\begin{center}\label{+tree}
\begin{forest} 
[
{$[\overline{1_2}]$\quad$[\overline{2_2}]$}
[
{$[\overline{2_2,1_2}]$} 
[
{$[\overline{2_2,1_4}]$} 
[
{$[\overline{2_2,1_6}]$}
[$\ldots$]
[$\ldots$]]
[
{$[\overline{2_2,1_2,2_2,1_4}]$}
[$\ldots$]
[$\ldots$]
]] 
[
{$[\overline{2_4,1_2}]$} 
[
{$[\overline{2_4,1_2,2_2,1_2}]$}
[$\ldots$]
[$\ldots$]] 
[
{$[\overline{2_6,1_2}]$}
[$\ldots$]
[$\ldots$]]]]]
\end{forest}
\end{center}

 where $b_n$ means that  $b$ is repeated $n$ times. We note that it is more convenient to write $[\overline{1_2}]$ instead of $[\overline{1}]$ in connection with  the conjunction operator in (\ref{conj}). The fact that all of the partial quotients of Markov quadratics are $1$ or $2$ and many other of their properties can be found in \cite{aigner}, \cite{bombieri}, \cite{Malyshev} and references therein. (See for example Corollary 1.27 in \cite{aigner}.)
 
Markov   numbers arise in many different contexts:  see  \cite{BGS}, \cite{BGS2}, \cite{GS} for some recent developments regarding the  Markov surfaces.
  \medskip
 
The main goal of this paper is to study  the values of modular functions  along the  tree associated to the Markov quadratics.

 Let $\Gamma=\PSL$. For a general quadratic irrationality $w\in\mathbb{Q}(\sqrt{D})$ and a modular function $f$ for $\Gamma $, the ``value" of $f$ at $w$ is defined in terms of the integral of $f$ along the  geodesic cycle $C_w\subset \Gamma\backslash\H $ associated to $w$. More precisely 
 
$$
f(w):= \int_{C_w}  f(z) ds
$$

where $ds$ is the hyperbolic arc length.
We can normalize the number $f(w)$ by the length of the geodesic $C_w$ and define $$
f^{nor}(w):=\dfrac{f(w)}{2\log \varepsilon_D}
$$
where $\varepsilon_D$ is the fundamental unit.(cf. section \ref{section-cycle}.)

The values of modular functions at real quadratic irrationalities   were introduced in \cite{DIT} and independently in \cite{Kaneko}. In \cite{DIT} their averages over ideal classes were shown to be coefficients of mock modular forms whereas   in \cite{Kaneko} Kaneko studied their  individual values $f^{nor}(w)$ (in the case   that the  modular function is the Klein's $j$ invariant), and based on numerical calculations he made several   interesting observations and conjectures. 

In this paper we  prove two of Kaneko's conjectures which involve  the values of modular functions at the Markov quadratics. 
Let $B$ be any branch of the Markov tree where with a branch we mean  a path on the tree without  any zigzags. Our first theorem shows that if  $w_n^B$ is the $n$-th Markov quadratic on a branch $B$ and $w_0^B$ is the predecessor of the tip of $B$ then the normalized values $f^{nor}(w^B_n),$ for any modular function $f,$   converge to the value $f^{nor}(w_0^B)$. (For more precise definitions of  the tip of a  branch and its predecessor see section \ref{section-markov}.) More precisely 

\begin{theorem}\label{convergence}
Let $f$ be a modular function defined on $\H$. For any branch $B$ of the Markov tree we have  $$
\lim
_{n\rightarrow\infty} f^{nor}(w^B_n) = f^{nor}(w^B_0).
$$
\end{theorem}

Our second theorem proves an eventual  monotonicity result which also partially proves the interlacing property of the values for the Markov quadratics that was conjectured by Kaneko;

\begin{theorem}\label{interlace} Let $f$ be a modular function on $\H$, let $B$ be any branch of the Markov tree. Then there exists a constant $N_{f,B}$ such that,  for all 
 $n\geq N_{f,B}$, the real and imaginary parts of $f^{nor}(w^B_{n+1})$ lie between the real and respectively  imaginary parts  of $f^{nor}(w^B_0)$ and $f^{nor}(w^B_n)$.
\end{theorem}

The rest of the paper is organized as follows. In the next section we give the  preliminaries about  cycle integrals and  continued fractions.  In section 3, we give the basic properties of the Markov quadratics   and the Markov tree.  In section 4 and 5 we study the values of modular functions on the Markov tree and    prove Theorem \ref{convergence} and Theorem \ref{interlace} respectively.

\noindent {\bf Acknowledgements:}  We  thank M. Kaneko and the referee for numerous and very  helpful comments  that improved our exposition. 
\section{Preliminaries}

\subsection{Cycle integrals}\label{section-cycle}

Let $w$ be a real quadratic irrationality and $\tilde{w}$ be its conjugate. 
$w$ and $\tilde{w}$ are the roots of a quadratic  equation
$$
ax^2+bx+c=0\qquad (a,b,c\in\Z, \quad(a,b,c)=1)
$$ 
with discriminant $D=b^2-4ac>0$.
We
change $[a,b,c]$ to $-[a,b,c]$ if necessary and write   $w=\frac{-b+\sqrt{D}}{2a}$, $\tilde{w}=\frac{-b-\sqrt{D}}{2a}$. 
 The geodesic $S_w$ in $\H$ joining $w$ and $\tilde{w}$ is given by the equation
$$
a|z|^2+b\, \Re(z)+c=0\qquad (z\in\H).
$$
The stabilizer
$\Gamma_w$ of $w$ in $\Gamma$  preserves the quadratic form $Q_w=[a,b,c]$, and hence $S_w$. The group $\Gamma_w$ is infinite cyclic; it   corresponds to the group $U_D^2$ of units of norm one of $\mathbb{Q}(\sqrt{D})$ via the isomorphism:
\begin{equation}\label{iso}
\begin{array}{ccc}
\Gamma_w &\longrightarrow &U_D^2\\
\begin{pmatrix} a&b\\c &d\end{pmatrix} &\mapsto &(a-cw)^2.
\end{array}
\end{equation}
We denote by $A_w$ the generator of $\Gamma_w$
$$A_w=\begin{pmatrix} \frac{1}{2}(t-bu) &-cu\\au &\frac{1}{2}(t+bu)\end{pmatrix},
$$
where $(t,u)$ is the smallest positive solution to Pell's equation $t^2-Du^2=4$, and we denote by $\varepsilon$  the generator of the infinite cyclic part of $U_D$ whose  square corresponds to $A_w$ by the isomorphism \eqref{iso}.

For any modular function $f$, since the group $\Gamma_w$   preserves the expression $f(z)Q_w(z,1)^{-1}dz,$   one can define the cycle integral of  $f$ along $C_w=S_w/\Gamma_w$, also viewed as the ``value" of $f$ at $w$, by the complex number
\begin{equation}\label{cycle-integral}
f(w):=\int_{C_w} \dfrac{\sqrt{D}f(z)}{Q_w(z,1)}dz.
\end{equation}
The factor $\sqrt{D}$ is introduced here for convenience but is also natural since  the constant function $f\equiv 1$, (\ref{cycle-integral}) gives the length of the geodesic $C_w.$
The integral defining $f(w)$   is $\Gamma$-invariant and can in fact be taken along any path in $\H$  from $z_0$ to $A_w^{-1}
 z_0$, where $z_0$ is any point  in $\H$.  
Note that this gives an orientation on  $S_w$ from $w$ to $\tilde{w}$,  which is  counterclockwise if $a>0$ and clockwise if $a<0$. 
We  normalize the number $f(w)$ by the length of the geodesic $C_w$ which is given by
$$
\int_{ C_w} \dfrac{\sqrt{D}}{Q_w(z,1)}dz=2\log \varepsilon
$$
  and we define the normalized value as 
$$
f^{nor}(w):=\dfrac{f(w)}{2\log \varepsilon}.
$$

\bigskip

\subsection{The `+' and `--' continued fractions}
 
 \smallskip

Let  $(b_0,b_1,b_2,\ldots)$  denote the `--' continued fraction
$$
(b_0,b_1,b_2,\ldots)=b_0-\dfrac{1}{b_1-\dfrac{1}{b_2-\dfrac{1}{\ddots}}}
$$
and   $[a_0,a_1,a_2,\ldots]$  be the `+' continued fraction
$$
[a_0,a_1,a_2,\ldots]=a_0+\dfrac{1}{a_1+\dfrac{1}{a_2+\dfrac{1}{\ddots}}}.
$$
Every real number $w$ has a `--' continued fraction expansion $w=(b_0,b_1,b_2,\ldots)$ with $b_i\in\Z$ and $b_i \geq 2$ for $i\geq 1$ and  a unique `+' continued fraction expansion $w=[a_0,a_1,a_2,\ldots]$ with $a_i\in\Z$ and $a_i \geq 1$ for $i\geq 1$.
The `--' continued fraction expansion of $w$ is obtained by setting $w_0=w$ and inductively $b_i=\ceil{w_i}$,  $w_{i+1}=\frac{1}{b_i-w_i}=ST^{-b_i}(w_i)$ , where $S(x)=-1/x$ and $T(x)=x+1$.
The `+' continued fraction expansion is obtained by setting  $a_i=\floor{w_i}$, $w_{i+1}=\frac{1}{w_i-a_i}=\varepsilon T^{-a_i}(w_i)$, where $\varepsilon(x)=1/x$. Hence the `--'  continued fraction is given by transformations of $\Gamma$ on the real line, whereas the `+' continued fraction corresponds to transformations of $\GL$. 
 To go from the `+' to the `--' continued fraction expansions, the general rule is 
\begin{equation}\label{+ to -}
[a_0,a_1,a_2,\ldots]\longrightarrow(a_0+1,\underbrace{2,\ldots,2}_{a_1-1},a_2+2,\underbrace{2,\ldots,2}_{a_3-1},a_4+2,\ldots).
\end{equation}


It is well known that a real number $w$ is a quadratic irrationality if and only if its `--' continued fraction expansion (or equivalently, its `+' continued fraction) is eventually periodic: $w=(b_0,b_1,\ldots,b_k,\overline{b_{k+1},\ldots,b_{k+r}})$   where the line over $b_{k+1},\ldots,b_{k+r}$ denotes the period.  We say that $w$ is \textit{purely periodic} when all the partial quotients  repeat.
It will be useful for the rest of the paper to remember the following statements:
\begin{enumerate}[(I)]
\item two quadratic irrationalities have the same `--' period if and only if they are $\Gamma$-equivalent;\\
\item $w$ has a purely periodic `--' continued fraction expansion if and only if $0<\tilde{w}<1<w$, where $\tilde{w}$ is  the conjugate of $w$;\\
\item if $w=(\overline{b_0,\ldots,b_r})$, then $\frac{1}{\tilde{w}}=(\overline{b_r,\ldots,b_0})$.
\end{enumerate}
 \smallskip
 
These statements and  more information about negative continued fractions can be found in  \cite[p. 126 ff]{Zagier-book}.

The following lemma gives an upper bound for the distance between two real numbers in terms of the number of first partial quotients for which they coincide.

\begin{lemma}\label{coincide}
If the `--' continued fraction expansions of $u$ and $v$ coincide in the first $r+1$ partial quotients and their `+' continued fraction expansions have only $1$'s and $2$'s, then 
$$
|u-v|\leq 10\left(\dfrac{2}{1+\sqrt{5}}\right)^{2r}.
$$
\end{lemma}

\begin{proof}
Let $u$ and $v$ be as in the statement of the lemma. Then one can see, by applying the rule \eqref{+ to -}, that also the `+' continued fraction expansions of $u$ and $v$ coincide in the first $r+1$ partial quotients.
Hence, if we set $a_0,\ldots, a_r$ to be those partial quotients, the rational number $\frac{p}{q}=[a_0,\ldots,a_r]$ is a convergent of both  $u$ and $v$. Then it is well known that
$$
\left|u-\dfrac{p}{q}\right|\leq\dfrac{1}{q^2},\qquad  \left|v-\dfrac{p}{q}\right|\leq\dfrac{1}{q^2}
$$
and
$$
q\geq  \frac{1}{\sqrt{5}}\left(\dfrac{1+\sqrt{5}}{2}\right)^{r}.
$$

Therefore,
$$
|u-v|\leq \left|u-\dfrac{p}{q}\right| + \left|v-\dfrac{p}{q}\right| \leq 10\left(\dfrac{2}{1+\sqrt{5}}\right)^{2r}.
$$
\end{proof}

\section{Markov Tree}


\subsection{Markov's quadratics}\label{section-markov}

Let $
\left\{m_i\right\}_{i=1}^\infty=\left\{1,2,5,13,29,34,89,169,194,\ldots\right\}
$ be the set of Markov numbers. 
As in the introduction, for each Markov number $m_i$, we let  
 $
 \theta_i=\dfrac{3m_i-2k_i+\sqrt{9m_i^2-4}}{2m_i}
$ be the Markov quadratic 
where $k_i$ is an integer that satisfies $a_ik_i\equiv b_i\pmod{m_i}$ and $(a_i,b_i,m_i)$ is a solution to \eqref{markovn} with $m_i$ maximal. Changing the representative for $k_i \mod m_i$ does not chnage the $\Gamma$ orbit of $\theta_i.$
%
In Markov's theory, only $\PGL$-equivalence classes are relevant, which implies that the order of $(a_i,b_i)$ does not matter. Since we need $\Gamma$-equivalence, which distinguishes non-real $f(\theta_i)$ and its conjugate, here the order of $(a_i,b_i)$ becomes relevant. We fix it so that  $\Im(f(w))>0$.

The Markov tree $\mathcal T$ associated to the Markov quadratics given in the introduction is   in terms of the `+' continued fractions. 
Since the 
 cycle integrals  are $\Gamma$ and not $\PGL$ invariant, we will rather work with the `--' continued fraction. 
 By following the rule \eqref{+ to -},
  the Markov tree $\mathcal T$ becomes in  the `--' continued fraction:
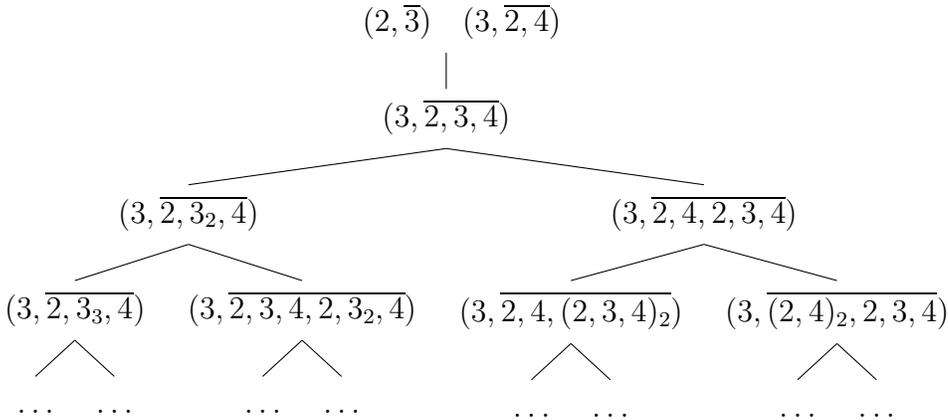
\begin{center}
\begin{forest}
[
\quad{$(2,\overline{3})$\quad$(3,\overline{2,4})$}
[
{$(3,\overline{2,3,4})$} 
[
{$(3,\overline{2,3_2,4})$} 
[
{$(3,\overline{2,3_3,4})$}
[$\ldots$]
[$\ldots$]]
[
{$(3,\overline{2,3,4,2,3_2,4})$}
[$\ldots$]
[$\ldots$]
]] 
[
{$(3,\overline{2,4,2,3,4})$} 
[
{$(3,\overline{2,4,(2,3,4)_2})$}
[$\ldots$]
[$\ldots$]] 
[
{$(3,\overline{(2,4)_2,2,3,4})$}
[$\ldots$]
[$\ldots$]]]]]
\end{forest}
\end{center}

Note that each branch (a path with no zigzags) in the tree $\mathcal T$ comes with a left or right orientation. 
We call a branch a left (right) branch if starting from its first vertex  on the top and going downwards the branch leans towards left (right).  
Since no zigzag paths are allowed, each branch has a unique orientation. For example, the branch with the quadratics $(3,\overline{2,3,4}),
(3,\overline{2,3_2,4}),(3,\overline{2,3_3,4})$ is a left branch, whereas the branch with $(3,\overline{2,3,4}),(3,\overline{2,4,2,3,4}),
(3,\overline{(2,4)_2,2,3,4})$ is a right branch.
We call the first vertex at  the top of any branch its tip. 
Except for the two singular cases of $(2,\overline{3})$ and  $(3,\overline{2,4})$, each Markov number lies both on a right and a left branch but it is the tip of only a left or a right branch, except for $(3,\overline{2,3,4})$ which is the tip of both the leftmost and the rightmost branches.

%

In the case of  `+' continued fractions  we consider a 
  \textit{conjunction operation} of two periods as
\begin{equation}\label{conj}
[\overline{s_0,\ldots,s_n}]\odot[\overline{t_0,\ldots,t_m}]=[\overline{s_0,\ldots,s_n,t_0,\ldots,t_m}].
\end{equation}
All Markov quadratics can be constructed by using this operation, starting with $[\overline{1}_2]$ and $[\overline{2}_2]$. Indeed, 
each Markov quadratic is the result of the conjunction operation of its predecessor on the same branch and the predecessor of the tip of the branch.

%
%
 
For the `--' continued fraction, 
the rule is also the conjunction of periods except for the most left branch, where the $n$-th Markov quadratic is $(3,\overline{2,3_n,4})$.
Indeed, let $x=[\overline{s_0,\ldots,s_n}]=(b_0,\overline{b_1,\ldots,b_k})$ and $y=[\overline{t_0,\ldots,t_m}]=(c_0,\overline{c_1,\ldots,c_\l})$. For any branch different from the   right most branch,
 by applying  \eqref{+ to -} together with  the observation that $s_n=t_m=1$ are in odd positions, so they do not contribute in the `--' expansion, we obtain 
$$
x\odot y=(b_0,\overline{b_1,\ldots,b_{k-1},t_0+2,c_1,\ldots,c_{\l-1},s_0+2}).
$$
But $t_0$ is equal to $1$ on the   left most branch and $2$ on any other branch, and $s_0=2$. For the   right most branch, \eqref{+ to -} also gives
$$
x\odot y=(b_0,\overline{b_1,\ldots,b_{k-1},4,c_1,\ldots,c_{\l-1},s_0+2})
$$
and $s_0=2$.

Throughout the paper, we denote by $w^B_n$ ($n\geq 1$) the $n$-th Markov quadratic on a branch $B$ of the tree and  $w_0^B $ the left (right) predecessor of the tip $w^B_1$ of $B$ if $B$ is a left (right) branch.  For example, if $B=L$ is the left most branch, then $w_0^L= (2,\overline{3})$, $w^L_1=(3,\overline{2,3,4})$, $w^L_2=(3,\overline{2,3_2,4})$, $w^L_3=(3,\overline{2,3_3,4})$, etc. If $B=R$ is the right most branch, then $w_0^R= (3,\overline{2,4})$, $w^R_1=(3,\overline{2,3,4})$, $w^R_2=(3,\overline{2,4,2,3,4})$, $w^R_3=(3,\overline{(2,4)_2,2,3,4})$, etc.

The $n$-th Markov quadratic on a left branch  $B\neq L$ can  be written as:
\begin{equation}\label{nth Markov}
w^B_n=(3,\overline{a_1,\ldots,a_s,(b_1,\dots,b_r)_n}),
\end{equation}
where $w_0^B=(3,\overline{b_1,\dots,b_r})$ and $a_1,...,a_s$ depend only on $B$. On a right branch $B$, we have
\begin{equation}\label{rightbranch}
w^B_n=(3,\overline{(b_1,\dots,b_r)_{n-1},a_1,\ldots,a_s}),
\end{equation}
and on the leftmost branch $L$ we have
\begin{equation}\label{leftmostbranch}
w^L_n=(3,\overline{2,3_n,4}).
\end{equation}

\begin{rk} The left most branch in the Markov tree is also called the Fibonacci branch since the associated  Markov numbers on this branch are the odd indexed Fibonacci numbers.
Similarly the right most branch is associated with the odd indexed Pell numbers  which are defined  by the recurrence $P_0=0, P_1=1$ and $P_{n+1}=2P_n+P_{n-1}.$ (cf. \cite{aigner} p. 49)\end{rk}


\subsection{The cycle of quadratics of a  Markov number} 

For    any quadratic irrationality $w,$  it is known that   the hyperbolic element $A_w$   is conjugate to a word in $T$ and $V$,
where 
$$
T=\begin{pmatrix}1 &1\\0 &1\end{pmatrix},\qquad V=\begin{pmatrix}1 &0\\1 &1\end{pmatrix}.
$$
If in particular $w=w^B_n$ is a quadratic on $\mathcal T$ ($n\geq 0$), then
the associated hyperbolic element $A_{w^B_n}$ can be written as a word in $T$ and $V$.
More specifically, $A_{w^B_n}=A_0^{-1} \ldots   A_{\l_n}^{-1}$, where $A_0=I$ and $A_i\in\{T^{-1},V^{-1}\}$ for $1\leq i\leq \l_n$ are given by  the  algorithm: 
$$
w^B_{n,0}=w^B_n,\qquad
 w^B_{n,i+1}=A_{i+1}(w^B_{n,i})\qquad (i\geq 0),
$$
where
$$
 A_{i+1}=\left\{\begin{array}{ll}
T^{-1} &\mbox{if $\floor{w^B_{n,i}}\geq 1 $},\\
V^{-1} &\mbox{otherwise}.
\end{array}\right.
$$
Hence
\begin{equation}\label{cycle}
w^B_{n,i}=A_i \ldots  A_0 w^B_{n},\qquad i=0,\ldots, \l_n,
\end{equation}
and $\l_n$ is the length of the word $A_{w^B_n}$, or equivalently, the length of the cycle of quadratics $\{w^B_{n,i}\}_i$ of $w^B_n$.  As the following example demonstrates   this procedure applied to a Markov quadratic in fact cycles back and hence terminates.

\begin{example}
For example,  the cycle of $w^L_1=(3,\overline{2,3,4})$ on the leftmost branch is:
$$
\begin{array}{ll}
w^L_{1,0}=(3,\overline{2,3,4})\\\\
w^L_{1,1}=T^{-1}(w^L_{1,0})=(2,\overline{2,3,4})\\\\
w^L_{1,2}=T^{-1}(w^L_{1,1})=(1,\overline{2,3,4})\\\\
w^L_{1,3}=V^{-1}(w^L_{1,2})=(1,\overline{3,4,2})\\\\
w^L_{1,4}=V^{-1}(w^L_{1,3})=(2,\overline{4,2,3})\\\\
w^L_{1,5}=T^{-1}(w^L_{1,4})=(1,\overline{4,2,3})\\\\
w^L_{1,6}=V^{-1}(w^L_{1,5})=(3,\overline{2,3,4})=w^L_{1,0}.
\end{array}
$$
The length is $\l_1=6$ and
$A_{w^L_1}=ITTVVTV.
$
 
\end{example}

From now on  we restrict to a left branch but not the leftmost branch. All the following arguments apply  in the same way if $B$ is a right branch or $B=L$, the leftmost branch. The small difference in the arguments arise due to the different conjunction operations necessary, which are given in (\ref{rightbranch}) for the right and in (\ref{leftmostbranch}) for   the leftmost branches. 

We  now consider $w^B_n$, in a left branch $B\neq L$, written as in \eqref{nth Markov}. Then
\begin{equation}\label{eqn-ell}
\l_n= n\l_0 + \sum_{i=1}^s (a_i-1),
\end{equation}
where
$$\l_0=\sum_{i=1}^r (b_i-1)$$
is the length of the cycle of $w^B_0$.
The number of partial quotients in the period of $w_1^B$ is $s+r$ and the conjunction operation ensures that  this  is $\leq 2r$. Hence  $s\leq r$ and since $a_i\leq 4$, we have 
  
 \begin{equation}\label{length}
 \l_n \leq 3r (n+1).
 \end{equation}

 It is convenient to set 

$$a=\sum_{i=1}^s(a_i-1)$$
and
$$
 \p=(b_1,\ldots,b_r),\qquad\p_k=(b_1,\ldots,b_r)_k,\qquad \q_k=(b_r,\ldots,b_1)_k,
 $$
where the subindex $k$ means that the continued fraction is repeated $k$ times.
 With these notations, the cycle of $w^B_n$  is of the form:
$$
\begin{array}{rl}
w^B_{n,0}=&(3,\overline{a_1,\ldots,a_s,\p_n}),\\\\
w^B_{n,1}=&(2,\overline{a_1,\ldots,a_s,\p_n}),\\\\
w^B_{n,2}=&(1,\overline{a_1,\ldots,a_s,\p_n}),\\\\
w^B_{n,3}=&(a_1-1,\overline{a_2,\ldots,a_s,\p_n,a_1}),\\
&\qquad\qquad\vdots\\
w^B_{n,a}=&(3,\overline{\p_n,a_1,\ldots,a_s}),\\
&\qquad\qquad\vdots\\
w^B_{n,a+\l_0}=&(3,\overline{\p_{n-1},a_1,\ldots,a_s,\p}),\\
&\qquad\qquad\vdots\\
w^B_{n,a+n\l_0}=&(3,\overline{a_1,\ldots,a_s,\p_n})=w^B_{n,0}.
\end{array}
$$

\begin{rk}\label{rk2}

 One can easily write the continued fraction expansion for  the Galois conjugate $-\tilde w^B_{n,i}$ of $-w^B_{n,i}$ in terms of that of $w^B_{n,i}$.  Indeed, let $(d_0,\overline{d_1,\ldots,d_m})$ be the continued fraction expansion of $w^B_{n,i}$. The  quadratic $ST^{-d_0}(w^B_{n,i})$ is purely periodic with continued fraction $(\overline{d_1,\ldots,d_m})$ so, by  the property (III),  its Galois conjugate is $1/(\overline{d_m,\ldots,d_1})$. Therefore, 
$$
\tilde w^B_{n,i}=T^{d_0}S(1/(\overline{d_m,\ldots,d_1}))=-(d_m-d_0,\overline{d_{m-1},\ldots,d_1,d_m}).
$$

  \end{rk}

\section{Convergence property}

In this section we study the values of a modular function on the Markov tree. Let $B$ be any branch of the tree and  $w^B_n$ be the $n$-th Markov quadratic on $B$.  Let  $A_{w^B_n}=A_0^{-1} \ldots   A_{\l_n}^{-1}$, where $A_0=I$ and $A_i\in\{T^{-1},V^{-1}\}$ for $1\leq i\leq \l_n.$ Let $\rho=e^{\pi i/3}$ and $z_i=A_0^{-1}\ldots A_i^{-1}\rho^2$.  Then using the modularity of $f$ we have 
\begin{align*}
f(w^B_n)&= -\sqrt{D} \sum_{i=0}^{\l_n-1} \int_{z_{i}}^{z_{i+1}} \dfrac{f(z)}{Q_{w^B_n}(z,1)}\, dz\\
&=- \sqrt{D} \sum_{i=0}^{\l_n-1} \int_{\rho^2}^{A^{-1}_{i+1} \rho^2} \dfrac{f(z)}{(Q_{w^B_n}|A_0^{-1} \ldots A_i^{-1})(z,1)}\, dz,\\
&= -\sum_{i=0}^{\l_n-1}\int_{\rho^2}^{A^{-1}_{i+1} \rho^2}f(z)\left(\dfrac{1}{z-w^B_{n,i}}-\dfrac{1}{z-\tilde w^B_{n,i}}\right) dz.
\end{align*}
 Since $V(\rho^2)=T(\rho^2)=\rho$,
  we obtain: 

\begin{lemma}\label{lemma2} For $n\geq 0$ we have
\begin{equation}\label{**}
f(w^B_n)= \int_{\rho}^{\rho^2} \sum_{i=0}^{\l_n-1} f(z)\left(\dfrac{1}{z-w^B_{n,i}}-\dfrac{1}{z-\tilde w^B_{n,i}}\right) dz.
\end{equation}
\end{lemma}

Lemma \ref{lemma2} is the main tool we use to estimate the values of modular functions at real quadratic irrationalities.

Throughout the paper, we denote by  $\C$  the arc of circle  joining $\rho^2$ and $\rho$. 
We denote by $\varepsilon^B_n$ the image of $A_{w_n}^B$ under the isomorphism \eqref{iso}, so the length of $C_{w_n^B}$ equals $2\log \varepsilon^B_n$.

Our first goal is to show that the normalized values $f^{nor}(w^B_n)$ for any modular function $f$ along any branch $B$ converge to the value $f^{nor}(w_0^B)$.  
 We call this property   `convergence property' and prove it  in this section. The main idea of  the proofs is to divide the sum in Lemma~\ref{lemma2} into several ranges and bound each piece making repeated use of Lemma~\ref{coincide}.
For simplicity of  the notation, as mentioned before,  we restrict to a left but not the leftmost branch. 
However, the argument in the proof of Theorem \ref{th1} applies in the same way if $B$ is a right branch or $B=L$. Only the bound $\delta_1(r,N)$ will be slightly modified but will still be of the form $O(rN\lambda^{rN})$ where $\lambda=\left(\frac{2}{1+\sqrt{5}}\right)^2.$ Hence Corollary  4.4 also remains true for any branch.

\begin{theorem}\label{th1} Let $f$ be a modular function,  $B$ be any left branch $\neq L$ of the Markov tree $\T$  and $N\geq 1$. There  exists a complex number $K=K_{f,B,N}$ such that for all $n\geq N$,
\begin{equation}\label{th1eq1}
|f(w^B_n)-n f(w^B_0)-K| \leq \delta_1(r,N)\max_{z\in\C}|f(z)|,
\end{equation}
where 
\begin{equation}\label{delta1}
\delta_1(r,N)=\dfrac{80\pi}{3}(2+r(N+1)) \left(\frac{2}{1+\sqrt{5}}\right)^{2(rN-1)}
\end{equation}
and $r+1$ is the number of partial quotients in the period  of $w^B_0$.
\end{theorem}

\begin{proof}

By applying  Lemma \ref{lemma2} for $f(w^B_n)$ and $f(w^B_0)$ we have:
\begin{equation}\label{eqsum}
f(w^B_n)-nf(w^B_0)=\int_{\rho}^{\rho^2}f(z) (S_1(n,N,z)+ S_2(n,N,z) + S_3(n,N,z))\, dz,
\end{equation}
where
\begin{align*}
S_1(n,N,z)&=
\sum_{i=0}^{a-1} \dfrac{1}{z-w^B_{n,i}} + \sum_{i=a+(n-N)\l_0}^{\l_n-1} \dfrac{1}{z-w^B_{n,i}} \\ 
&-\sum_{i=0}^{a+N\l_0-1}\dfrac{1}{z-\tilde w^B_{n,i}}
 - N\sum_{i=0}^{\l_0-1}\left(\dfrac{1}{z-w_{0,i}^B}-\dfrac{1}{z-\tilde w^B_{0,i}}\right),\\
S_2(n,N,z)&= \sum_{i=a}^{a+(n-N)\l_0-1} \dfrac{1}{z-w^B_{n,i}}
-(n-N)\sum_{i=0}^{\l_0-1} \dfrac{1}{z-w^B_{0,i}},\\
S_3(n,N,z)&= - \sum_{i=a+N\l_0}^{a+n\l_0-1} \dfrac{1}{z-\tilde w^B_{n,i}}+ (n-N)\sum_{i=0}^{\l_0-1}\dfrac{1}{z-\tilde w^B_{0,i}}.
\end{align*}

Moreover, we can also write
\begin{equation}\label{dec2}
 S_1(n,N,z) =  S_1(N,N,z) + (S_1(n,N,z)-S_1(N,N,z)).
\end{equation}

Define
$$
K:=\int_{\rho}^{\rho^2} f(z)S_1(N,N,z)dz
$$
and
$$c(n,z):=   |S_1(n,N,z)-S_1(N,N,z)| + |S_2(n,N,z)|+ | S_3(n,N,z)|.
$$
Then
\begin{equation}\label{K}
|f(w^B_n)-nf(w^B_0)-K| \leq \int_{\rho}^{\rho^2}  c(n,z) |f(z)||dz|.
\end{equation} 

These divisions are guided by the continued fraction expansions of all the terms in the cycle of $w_n^B$ and $w_0^B$ and their conjugates. As we will see shortly the repeated use of Lemma~\ref{coincide} will allow us to bound all the other sums after we separate the main term $K$.

Let $\lambda=\left(\frac{2}{1+\sqrt{5}}\right)^2$. If we can show that 
\begin{equation}\label{c}
c(n,z) \leq 80(2+r(N+1)) \lambda^{rN-1}
\end{equation}
 for $z\in\C$, then  the theorem is proved. Next we show \eqref{c}.
\\

\textbf{Bound for $|S_2(n,N,z)| $.}  We have that

\begin{align*}
|&S_2(n,N,z)|\leq 
  \sum_{k=0}^{n-N-2}
\sum_{i=1}^{\l_0} \dfrac{|w^B_{n,2+a+k\l_0+i}-w^B_{0,2+i}|}{|z-w^B_{n,2+a+k\l_0+i}||z-w^B_{0,2+i}|}\\
&+\sum_{i=0}^2\dfrac{|w^B_{n,a+i}-w^B_{0,i}|}{|z-w^B_{n,a+i}|
|z-w^B_{0,i}|}+
\sum_{i=1}^{\l_0-3} \dfrac{|w^B_{n,2+a+(n-N-1)\l_0+i}-w^B_{0,2+i}|}{|z-w^B_{n,2+a+(n-N-1)\l_0+i}||z-w^B_{0,2+i}|}.
\end{align*}

  Clearly for any $z\in\C$ and   $x\in\R,$ we have that $|z-x|\geq \Im(e^{2\pi i/3})=\sqrt{3/2}$.   Hence 
the denominators are bounded below by $\frac{3}{4}$ when $z\in\C$ since the points $w$ are real. The numerators can be bounded by using Lemma \ref{coincide}. For $i=0,1,2$, 
$$
w^B_{n,a+i}=(3-i,\overline{\p_n,a_1,\ldots,a_s})
$$ 
and
$$w^B_{0,i}=(3-i,\overline{\p})$$
coincide at least in the first $rn+1$ partial quotients. For each $0\leq k\leq n-N-2$,  we have: 

for $1\leq i\leq b_1-1$, 

\begin{equation}\label{b1}
w^B_{n,2+a+k\l_0+i}=(b_1-i,\overline{b_2,\ldots,b_r,\p_{n-1-k},a_1,\ldots,a_s,\p_k,b_1});
\end{equation}

for the next $b_2-1$ values of $i$ ($b_1\leq i\leq b_1+b_2-2$),
\begin{equation}\label{b2}
w^B_{n,2+a+k\l_0+i}=(b_2-j,\overline{b_3,\ldots,b_r,\p_{n-1-k},a_1,\ldots,a_s,\p_k,b_1,b_2})
\end{equation}
with $1\leq j\leq b_2-1$;  this process goes on until the last  $b_r-1$ values of $i$, where
$$
w^B_{n,2+a+k\l_0+i}=(b_r-j,\overline{\p_{n-1-k},a_1,\ldots,a_s,\p_{k+1}})
$$
with $1\leq j\leq b_r-1$.
For  $k=n-N-1$, we have the same pattern as before except for the last block of values of $i$, where we only have  $b_r-3$ of them.

Now, for each $0\leq k\leq n-N-1$, for $1\leq i \leq b_1-1,$   \eqref{b1}  and
$$
w^B_{0,2+i}=(b_1-i,\overline{b_2,\ldots,b_r,b_1})
$$
coincide   in the first $rn-rk$ partial quotients. The next $b_2-1$ values of $i$    \eqref{b2}  and  
$$
w^B_{0,2+i}=(b_2-j,\overline{b_3,\ldots,b_r,b_1,b_2}) \qquad(1\leq j\leq b_2-1)
$$
coincide in the first $rn-rk-1$ partial quotients, similarly for the next $b_3-1$  values of $i,$ $w^B_{n,2+a+k\l_0+i}$ and $w^B_{0,2+i}$ coincide in the first $rn-rk-2$ partial quotients, etc.  
Therefore, using Lemma~\ref{coincide},   for $z\in\C$, we have
\begin{align}\label{boundS2}
|S_2(n,N,z)|  \nonumber
&\leq \dfrac{40}{3}\left(3\lambda^{rn}+ \sum_{i=1}^r (b_i-1)\sum_{k=0}^{n-N-1}\lambda^{r(n-k)-i}\right)\\ \nonumber 
&\leq\dfrac{40}{3}\left(3\lambda^{rn}+3\left(\sum_{i=1}^r \lambda^{-i}\right)\left(\sum_{k=N+1}^{n} \lambda^{rk}\right)\right)\\ \nonumber
&\leq  \dfrac{40}{3}\left(3\lambda^{rn}+  
3\left(\sum_{i=1}^r \lambda^{r-i}\right)\left(\sum_{k=N}^{n-1} \lambda^{rk}\right)\right)\\ \nonumber
&\leq  \dfrac{40}{3}\left(3\lambda^{rn}+  
3\left(\sum_{i=0}^{r-1}\lambda^{i}\right)\left(\sum_{k=N}^{n-1} \lambda^{rk}\right)\right)\\ \nonumber
&\leq 40\lambda^{rN}\left(1+\dfrac{1}{(1-\lambda)(1-\lambda^r)}\right)\\ 
&\leq 120 \lambda^{rN}.
\end{align}

In the second inequality we used that $b_i\leq 4$, whereas  the last inequality follows from the numerical value $1/1-\lambda=1.618...$

\medskip

\textbf{Bound for $|S_3(n,N,z)|$.}
In a similar way we bound $|S_3(n,N,z)|$. We have that
\begin{align*}
|&S_3(n,N,z)|\leq  \sum_{k=N}^{n-2}
\sum_{i=1}^{\l_0} 
\dfrac{|\tilde w^B_{n,2+a+k\l_0+i} - \tilde w^B_{0,2+i}|}{|z-\tilde w^B_{n,2+a+k\l_0+i}| |z-\tilde w^B_{0,2+i}|}\\
&+\sum_{i=0}^2 
\dfrac{|\tilde w^B_{n,a+N\l_0+i} - \tilde w^B_{0,i}|}{|z-\tilde w^B_{n,a+N\l_0+i}| |z-\tilde w^B_{0,i}|}
+ \sum_{i=1}^{\l_0-3} \dfrac{|\tilde w^B_{n,2+a+(n-1)\l_0+i} - \tilde w^B_{0,2+i}|}{|z-\tilde w^B_{n,2+a+(n-1)\l_0+i}| |z-\tilde w^B_{0,2+i}|}.
\end{align*}

For $i=0,1,2$, using Remark \ref{rk2}, we have that  
$$
-\tilde w^B_{n,a+N\l_0+i}=(1+i,\overline{b_{r-1},\ldots,b_1,\q_{N-1},a_s,\ldots,a_1,q_{n-N},b_r})
$$ 
and
$$-\tilde w^B_{0,i}=(1+i,\overline{b_{r-1},\ldots,b_1,b_r})$$
coincide   in the first $rN$ partial quotients.  For each $N\leq k\leq n-2$,  we have: 

for $1\leq i\leq b_1-1$,
\begin{equation}\label{cb1}
-\tilde w^B_{n,2+a+k\l_0+i}=(i,\overline{\q_k,a_s\ldots,a_1,\q_{n-k}});
\end{equation}
for the next $b_2-1$ values of $i$ ($b_1\leq i\leq b_1+b_2-2$),
\begin{equation}\label{cb2}
-\tilde w^B_{n,2+a+k\l_0+i}=(j,\overline{b_1,\q_k,a_s\ldots,a_1,\q_{n-1-k},b_r,\ldots,b_3,b_2})
\end{equation}
with $1\leq j\leq b_2-1$; this process goes on until the last  $b_r-1$ values of $i$, where
$$
-\tilde w^B_{n,2+a+k\l_0+i}=(j,\overline{b_{r-1},\ldots,b_1,
\q_k,a_s,\ldots,a_1,\q_{n-1-k},b_r})
$$
with $1\leq j\leq b_r-1$.

For  $k=n-1$, we have the same pattern as before except for the last block of values of $i$, where we only have  $b_r-3$ of them.
Now, for each $N\leq k\leq n-1$, for the first $b_1-1$ values of $i$  \eqref{cb1} coincide with 
$$
-\tilde w^B_{0,2+i}=(i,\overline{\q})
$$
in the first $rk+1$ partial quotients. For the next $b_2-1$  values of $i$ \eqref{cb2} coincide with 
$$
-\tilde w^B_{0,2+i}=(j,\overline{b_1,b_r,\ldots,b_2}) \qquad(1\leq j\leq b_2-1)
$$
in the first $rk+2$ partial quotients,  for the next $b_3-1$ $i$-values $-\tilde w^B_{n,2+a+k\l_0+i}$ and  $-\tilde w^B_{0,2+i}$ coincide in the first $rk+3$ partial quotients, etc.  
Once again using Lemma\ref{coincide}, and the fact that $b_i\leq 4$ together with the numerical value of $\lambda,$  we have, for $z\in\C$,  
\begin{align}\label{boundS3}
|S_3(n,N,z)|\nonumber
&\leq  
\dfrac{40}{3}\left(3\lambda^{rN-1}+
\sum_{i=1}^r (b_i-1)\sum_{k=N}^{n-1}\lambda^{rk+i-1}\right)\\ \nonumber
&\leq 
\dfrac{40}{3}\left(3\lambda^{rN-1}+ 3\left(\sum_{i=1}^r  \lambda^{i-1}\right)
\left(\sum_{k=N}^{n-1}\lambda^{rk}\right)\right) \\ \nonumber
&\leq 40\lambda^{rN-1}\left(1+\dfrac{\lambda}{(1-\lambda)(1-\lambda^r)}\right)\\ 
&\leq 
80\lambda^{rN-1}.
\end{align}

\textbf{Bound for $|S_1(n,N,z)-S_1(N,N,z)|$}. 
We have
\begin{align*}
&|S_1(n,N,z)-S_1(N,N,z)| \leq 
\sum_{i=0}^{a-1} \dfrac{|w^B_{n,i}-w^B_{N,i}|}{|z-w^B_{n,i}||z-w^B_{N,i}|}\\
& + \sum_{i=a}^{\l_{N}-1} \dfrac{|w^B_{n,i+(n-N)\l_0}-w^B_{N,i}|}{|z-w^B_{n,i+(n-N)\l_0}||z-w^B_{N,i}|} 
+\sum_{i=0}^{a+N\l_0-1}\dfrac{|\tilde w^B_{n,i}-\tilde w^B_{N,i}|}{|z-\tilde w^B_{n,i}||z-\tilde w^B_{N,i}|}.
\end{align*}
Again the denominators are bounded below by $\frac{3}{4}$ for $z\in\C$  and we use Lemma \ref{coincide} to bound the numerators.
For the first term in the first sum, using 
\begin{equation}\label{fg}
w^B_{n,0}=(3,\overline{a_1,\ldots,a_s,\p_n})
\end{equation}
and
\begin{equation}\label{fg1}
w^B_{N,0}=(3,\overline{a_1,\ldots,a_s,\p_{N}}),
\end{equation}
one can see that the successive terms $w^B_{n,i}$ and $w^B_{N,i}$ (up to $i=a-1$)  coincide at least  in the first $rN$ partial quotients. This is also true for the second sum, where we have
$$
w^B_{n,a+(n-N)\l_0}=(3,\overline{\p_{n-N},a_1,\ldots,a_s,\p_{N}})
$$
and
$$
w^B_{N,a}=(3,\overline{\p_N,a_1,\ldots,a_s}),
$$
as well as for the third and fourth sums, where we can  use Remark \ref{rk2} and  the continued fractions of    
\eqref{fg} and \eqref{fg1}, and  
$$
w^B_{n,a+n\l_0}=(3,\overline{a_1,\ldots,a_s,\p_n}),
$$
and 
$$
w^B_{n,a+N\l_0}=(3,\overline{a_1,\ldots,a_s,\p_N}).
$$
respectively.
Hence, using (\ref{length}), we have 
\begin{equation}\label{boundS1}
|S_1(n,N,z)-S_1(N,N,z)|   \leq \dfrac{80}{3} \l_N \lambda^{rN-1} \stackrel{\eqref{length}}{\leq} 80r(N+1)\lambda^{rN-1}.
\end{equation}

Finally,   since $\lambda<2/3$, the bounds (\ref{boundS2}), (\ref{boundS3})  and (\ref{boundS1}) give 
$$
c(n,z) \leq 80(2+r(N+1)) \lambda^{rN-1}.
$$
\end{proof}

In particular, Theorem \ref{th1} applied to  the function $f=1$ gives:

\begin{coro}\label{coroth1}
Let $B$ be any left branch $\neq L$ of $\T$ and $N\geq 1$. For all $n\geq N$,
there exists $K=K_{B,N}\in\R$ such that
\begin{equation}\label{th1eq3}
|\log \varepsilon^B_n - n\log \varepsilon^B_0-K|\leq \delta_1(r,N)
\end{equation} 
with $\delta_1(r,N)$ and $r$ as in \eqref{delta1}.
\end{coro}

The next corollary proves Theorem~\ref{convergence} from the introduction.

\begin{coro}\label{coro}
Let $f$ be a modular function. For any left branch $B\neq L$ of $\T$,
$$
\lim_{n\rightarrow\infty} f^{nor}(w^B_n) = f^{nor}(w^B_0).
$$
\end{coro}

\begin{proof}

It follows from Theorem \ref{th1} and Corollary \ref{coroth1} that $|f(w^B_n)- n f(w^B_0)|$ and $|\log \varepsilon^B_n - n \log \varepsilon^B_0|$ are bounded above and below by absolute constants (not depending on $n$). Then
\begin{align*} 
0= \lim_{n\rightarrow\infty}  \dfrac{|f(w^B_n)-n f(w^B_0)|}{\log \varepsilon^B_n} 
=\lim_{n\rightarrow\infty}\left|\dfrac{f(w^B_n)}{\log\varepsilon^B_n}-\dfrac{f(w^B_0)}{\log\varepsilon^B_0}\right|.
\end{align*}

\end{proof}

\section{Interlacing property}

In this section we prove Theorem~\ref{interlace}. As in the proof of the convergence property we restrict again to a left but not the leftmost branch in what follows.The argument applies in the same way to any branch, with the bound $\delta_2(n,r)$ slightly modified. It will still be of the form $O(rn\lambda^{rn})$. Hence Theorem \ref{interlace} applies in fact to any branch of the Markov tree and it is a consequence of the next theorem whose proof is similar to the proof of Theorem \ref{th1}.

\begin{theorem}\label{claim2} Let $f$ be a modular function. For every left branch
 $B\neq L$ of the Markov tree $\T$ and for all $n\geq 1$,
\begin{equation}\label{cleq}
|f(w^B_{n+1})-f(w^B_n)-f(w^B_0)|\leq  \delta_2(n,r) \max_{z\in\C}|f(z)|
\end{equation}
where
\begin{equation}\label{delta2}
\delta_2(n,r)= \dfrac{80\pi}{3}(n+2)r\left(\frac{2}{1+\sqrt{5}}\right)^{2(rn-1)}
\end{equation}
and $r+1$ is the number of partial quotients in the period  of $w_0^B$.
\end{theorem}

\begin{proof}
Once again,  applying Lemma~\ref{lemma2} and \eqref{eqn-ell} gives 

\begin{align}\label{eqcl1}
f(w^B_{n+1})-f(w^B_n) &= f(w^B_0) + \int_ \rho^{\rho^2}f(z) R_1(n,z)\, dz  + \int_\rho^{\rho^2}f(z) R_2(n,z) \, dz,
\end{align} 
where
\begin{align*}
R_1(n,z)=&\sum_{i=0}^{a-1} \left(\dfrac{1}{z-w^B_{n+1,i}}-\dfrac{1}{z-w^B_{n,i}}\right) + \sum_{i=a}^{\l_n-1} \left(\dfrac{1}{z-w^B_{n+1,\l_0+i}}-\dfrac{1}{z-w^B_{n,i}}\right)\\
&- \sum_{i=0}^{\l_n-1} \left(\dfrac{1}{z-\tilde w^B_{n+1,i}}-\dfrac{1}{z-\tilde w^B_{n,i}}\right),\\
R_2(n,z)=&\sum_{i=0}^{\l_0-1} \left(\dfrac{1}{z-w^B_{n+1,a+i}}-\dfrac{1}{z-w^B_{0,i}} 
- \dfrac{1}{z-\tilde w^B_{n+1,\l_n+i}}+\dfrac{1}{z-\tilde w^B_{0,i}}\right).
\end{align*}

Next we give upper bounds for the norms of the two sums above when $z\in\C$. We set again $\lambda=\left(\frac{2}{1+\sqrt{5}}\right)^2$.
\\

\textbf{Bound for $|R_1(n,z)|$.} For $z\in\C$, we have
\begin{align*}
|R_1(n,z)| &\leq \sum_{i=0}^{a-1} \dfrac{|w^B_{n+1,i}-w^B_{n,i}|}{|z-w^B_{n+1,i}||z-w^B_{n,i}|} + \sum_{i=a}^{\l_n-1} \dfrac{|w^B_{n+1,i+\l_0}-w^B_{n,i}|}{|z-w^B_{n+1,i+\l_0}||z-w^B_{n,i}|}\\
&+ \sum_{i=0}^{a+n\l_0-1} \dfrac{|\tilde w^B_{n+1,i}-\tilde w^B_{n,i}|}{|z-\tilde w^B_{n+1,i}||z-\tilde w^B_{n,i}|}+
\sum_{i=a+n\l_0}^{\l_n-1}\dfrac{|\tilde w^B_{n+1,i}-\tilde w^B_{n,i}|}{|z-\tilde w^B_{n+1,i}||z-\tilde w^B_{n,i}|}.
\end{align*}
As before we use the bound of $\frac{3}{4}$ for  the denominators  and Lemma \ref{coincide} for the numerators.
In the first sum using 
\begin{equation}\label{n+1,i}
w^B_{n+1,0}=(3,\overline{a_1,\ldots,a_s,\p_{n+1}})
\end{equation}
and
\begin{equation}\label{n,i}
w^B_{n,0}=(3,\overline{a_1,\ldots,a_s,\p_n}),
\end{equation}
one can see that the  successive terms $w^B_{n+1,i}$ and $w^B_{n,i}$   (up to $i=a-1$) coincide at least  in the first $rn$ partial quotients.
The same is true for the second sum, where  
$$
w^B_{n+1,a+\l_0}=(3,\overline{\p_n,a_1,\ldots,a_s,\p})
$$
and
$$
w^B_{n,a}=(3,\overline{\p_n,a_1,\ldots,a_s}).
$$
For the third and fourth sums,   we use once again  Remark \ref{rk2} together with      
\eqref{n+1,i} and \eqref{n,i},   and  
$$
w^B_{n+1,a+n\l_0}=(3,\overline{a_1,\ldots,a_s,\p_{n+1}})
$$
and
$$
w^B_{n,a+n\l_0}=(3,\overline{a_1,\ldots,a_s,\p_n})
$$
respectively.

Hence
$$
|R_1(n,z)|\leq \dfrac{80}{3}\l_n\lambda^{rn-1} \stackrel{\eqref{length}}{\leq} 80r(n+1) \lambda^{rn-1}.
$$

\textbf{Bound for $|R_2(n,z)|$.}  In a similar way we bound this second sum when $z\in\C$:
$$
|R_2(n,z)| \leq\sum_{i=0}^{\l_0-1} \dfrac{|w^B_{n+1,a+i}-w^B_{0,i}|}{|z-w^B_{n+1,a+i}||z-w^B_{0,i}|}
+ \dfrac{|\tilde w^B_{n+1,\l_n+i}-\tilde w^B_{0,i}|}{|z-\tilde w^B_{n+1,\l_n+i}||z-\tilde w^B_{0,i}|}.
$$ 
Again using 
$$
w^B_{n+1,a}=(3,\overline{\p_{n+1}})
$$
and
\begin{equation}\label{0,0}
w^B_{0,0}=(3,\overline{\p}),
\end{equation}
one can see that all the successive terms $w^B_{n+1,a+i}$ and $w^B_{0,i}$ in the sum coincide at least in the first $rn$ partial quotients.
For the conjugate terms, one can see from   \eqref{0,0} and 
$$
w^B_{n+1,\l_n}=(3,\overline{\p,a_1,\ldots,a_s,\p_n})
$$
 that  $-\tilde w^B_{n+1,\l_n+i}$ and $-\tilde w^B_{0,i}$ coincide as well in the first $rn$ partial quotients.
  Hence
\begin{align*}
|R_2(n,z)| \leq \dfrac{80}{3}\l_0\lambda^{rn-1} \stackrel{\eqref{length}}{\leq} 80r \lambda^{rn-1}.
\end{align*}

Therefore,
\begin{align*}
|f(w_{n+1})-f(w_n)-f(w_0)|&\leq \int_{\rho}^{\rho^2} |f(z)| (|R_1(n,z)|+|R_2(n,z)|)\, |dz|\\
&\leq \delta_2(n,r) \max_{z\in\C}|f(z)|
\end{align*}
with $$\delta_2(n,r)=\dfrac{80\pi}{3} r(n+2)\lambda^{rn-1}.$$

\end{proof}

  Theorem \eqref{claim2} applied to the function $f=1$ gives:

\begin{coro}\label{coroclaim2}
For every left branch $B\neq L$ of $\T$ and for all $n\geq 1$,
$$
|\log \varepsilon_{n+1}^B-\log \varepsilon_n^B-\log \varepsilon_0^B|\leq \delta_2(n,r)
$$
\end{coro}
with $\delta_2(n,r)$ and $r$ as in \eqref{delta2}.

\bigskip 
We finish this section  by giving the proof of Theorem~\ref{interlace} in the case that the branch $B$ is any left branch $\neq L.$ The proof of the general case goes along the same lines.

\begin{theorem}\label{th2} Let $f$ be a modular function,  $B$ be any left branch $\neq L$ of the Markov tree $\T$. There exists a constant $N_{f,B}$ such that,  for all 
 $n\geq N_{f,B}$, the real and imaginary parts of $f^{nor}(w^B_{n+1})$ lie between the real and imaginary parts respectively of $f^{nor}(w^B_0)$ and $f^{nor}(w^B_n)$.
\end{theorem}

\begin{proof}

By definition, the inequality
$$
\Re(f^{nor}(w^B_{n}))<\Re(f^{nor}(w^B_0))
$$
holds if and only if
\begin{equation}\label{insandwich0}
 \Re(f(w^B_{n})) \log\varepsilon_0^B < \Re(f(w^B_0))\log\varepsilon_{n}^B.
\end{equation}

 Let $N,M$ be positive constants. For all $n\geq \max(N,M)$, we can write
\begin{equation}\label{w1}
\Re(f(w_n^B))= n\Re(f(w_0^B)) + K_{f,B,N} + \varepsilon_1(n,N),
\end{equation}
\begin{equation}\label{w2}
\log \varepsilon_n^B=n\log \varepsilon_0^B+ K_{1,B,M} +\varepsilon_2(n,M)
\end{equation}
where $K_{f,B,N}$, $K_{1,B,M}$ are the real parts of the constants in Theorem
\ref{th1} and Corollary \ref{coroth1} respectively, 
$|\varepsilon_1(n,N)|\leq \delta_1(N)\max_{z\in\C}|f(z)|$ and $|\varepsilon_2(n,M)|\leq \delta_1(M)$.
Therefore \eqref{insandwich0}
is equivalent to 
\begin{equation}\label{awful}
\Re(f^{nor}(w^B_0)) >\dfrac{K_{f,B,N}}{K_{1,B,M}}+\dfrac{\varepsilon_1(n,N)-\varepsilon_2(n,M)\Re(f^{nor}(w^B_0))}{K_{1,B,M}}.
\end{equation} 
There exists a constant $C_1(f,B)$ depending on $f$ and $B$ such 
that, for $\max(N,M)\geq C_1(f,B)$, \eqref{awful} is equivalent to either
\begin{equation}\label{awful2}
\Re(f^{nor}(w^B_0)) > \dfrac{K_{f,B,N}}{K_{1,B,M}}
\end{equation} 
or \eqref{awful2} with the strict inequality replaced by $\geq$, according to whether the error term in \eqref{awful} is positive or negative. 
If we can choose $N,M\geq C_1(f,B)$ satisfying $\Re(f^{nor}(w^B_0))\neq \frac{K_{f,B,N}}{K_{1,B,M}}$, then 
\eqref{awful} is equivalent to
\eqref{awful2} for those $N,M$. If we cannot choose such $N,M$, then $K_{f,B,N}, K_{1,B,M}$ would be constants that do not depend on $N,M$, and in particular
$\varepsilon_1(n,N)=\varepsilon_2(n,M)=0$. Hence, also in this case \eqref{awful} is equivalent to \eqref{awful2} for all $N,M\geq C_1(f,B)$.

In a similar way, the inequality 
$$
\Re(f^{nor}(w^B_{n}))>\Re(f^{nor}(w^B_0))
$$
is equivalent to
\begin{equation}\label{awful3}
\Re(f^{nor}(w^B_0)) < \dfrac{K_{f,B,N}}{K_{1,B,M}}
\end{equation}
for $N,M$ chosen as before.
 Since
 \eqref{awful2} and \eqref{awful3} do not depend on $n$, 
we have  either
$$
 \Re(f^{nor}(w_{n}^B)) <\Re(f^{nor}(w_0^B))
$$
simultaneously for all $n\geq \max(N,M)$ with $N,M$ chosen as before,
or 
$$
 \Re(f^{nor}(w_{n}^B)) >\Re(f^{nor}(w_0^B)).
$$

Similarly, the inequality
$$
\Re(f^{nor}(w^B_{n+1}))<\Re(f^{nor}(w^B_n))
$$
holds if and only if
\begin{equation}\label{insandwich1}
 \Re(f(w^B_{n+1})) \log\varepsilon_n^B < \Re(f(w^B_n))\log\varepsilon_{n+1}^B.
\end{equation}

 Theorem \ref{claim2} and Corollary \ref{coroclaim2} respectively imply that 
$$
\Re(f(w^B_{n+1}))=\Re(f(w^B_n))+\Re(f(w_0^B))+\mu(n)$$
with $|\mu(n)|\leq \delta_2(n)\max_{z\in\C}|f(z)|$  and  
$$
\log\varepsilon^B_{n+1}=\log\varepsilon_n^B+\log\varepsilon_0^B +\nu(n)
$$
with $|\nu(n)|\leq \delta_2(n)$. Hence 
\eqref{insandwich1} is equivalent to
 \begin{equation}\label{insandwich2}
  ( \Re (f(w^B_0)) + \mu(n)) \log\varepsilon_n^B < \Re (f(w_n^B))  (\log\varepsilon_0^B + \nu(n)).
  \end{equation}
  Now, there exists a constant $C_2(f,B)\geq C_1(f,B)$  such that, for $n\geq C_2(f,B)$, we have that $\Re(f^{nor}(w^B_{n}))\neq\Re(f^{nor}(w^B_0))$ and that 
   \eqref{insandwich2} is equivalent to 
 \begin{equation}\label{insandwich3}
\Re(f(w^B_{0})) \log\varepsilon_n^B  <  \Re(f(w^B_{n})) \log\varepsilon_0^B.
 \end{equation}
 
  Using \eqref{w1} and \eqref{w2} again, we obtain that \eqref{insandwich3}
is equivalent to 
\begin{equation}\label{awful1}
\Re(f^{nor}(w^B_0)) < \dfrac{K_{f,B,N}}{K_{1,B,M}}
\end{equation} 
where $N,M$ are chosen as before.

Therefore,
we finally have that either
$$
\Re(f^{nor}(w_0^B)) < \Re(f^{nor}(w_{n+1}^B)) <\Re(f^{nor}(w_n^B)) 
$$
for all $n\geq \max(C_2(f,B),N,M)$
or 
$$
\Re(f^{nor}(w_n^B)) <\Re(f^{nor}(w_{n+1}^B)) < \Re(f^{nor}(w_0^B)).
$$

The same argument applies to the imaginary parts of $f^{nor}(w^B_{n+1})$, $f^{nor}(w^B_n)$ and $f^{nor}(w^B_0)$.

\end{proof}

\end{document}